\documentclass[a4paper,10pt]{article}
\usepackage[english]{babel}
\usepackage[utf8]{inputenc}
\usepackage[left=3.2cm, right=3.2cm, top=4cm, bottom=4cm]{geometry}

\usepackage[nottoc]{tocbibind}
\usepackage{amsthm}
\usepackage{bm}
\usepackage{hyperref}
\usepackage{enumerate}
\usepackage{graphicx, psfrag, amscd, amssymb,amsmath,amsthm, amsfonts}
\usepackage{mathrsfs}
\usepackage{cite}
\usepackage{titling}
\predate{}
\postdate{}

\newtheorem{theorem}{Theorem}[section]           % theorems
\newtheorem{corollary}[theorem]{Corollary}       % corollaries
\newtheorem{lemma}[theorem]{Lemma}               % lemmas
\newtheorem{conjecture}[theorem]{Conjecture}     % conjecture
\theoremstyle{definition} 
\newtheorem{definition}[theorem]{Definition}     % definitions
\theoremstyle{remark}
\newtheorem{remark}[theorem]{Remark}             % remarks

\newcommand{\te}{\textrm}

\title{Topics on Pedal Polygons}
\author{Chia-An Hsu, Hsin-Chuang Chou, Chen-Rui Liu,\\ Chih-Hsuan Liang, and Yu-Wei Chang}
\date{}

\begin{document}

\maketitle

\abstract{In this paper, we study several topics on pedal polygons. First, we prove the existence for pedal centers of triangles in a new way. From its proof, we find that the sum of area of outer and inner polygons is invariant under rotation. Finally, we investigate when a polygon will be a pedal polygon of another one.}

%\abstract{In this paper, we investigate pedal polygons. We prove the existence for pedal centers of triangles in a new way. Also, we generalize the proof and find that area is an invariant in this situation. Finally, we use the construction of pedal polygons to define an equivalence relation, and study the equivalence classes.}

%statistics on alternating words under correspondence between ``possible reflection paths within several layers of glass'' and ``alternating words''.  For $v=(v_1,v_2,\cdots,v_n)\in\mathbb{Z}^{n}$, we say $P$ is a path within $n$ glass plates corresponding to $v$, if $P$ has exactly $v_i$ reflections occurring at the $i^{\rm{th}}$ plate for all $i\in\{1,2,\cdots,n\}$.

\section{Introduction}
There are several results about pedal polygons. For example, for any two triangles $V$ and $W$, there exists a point $X$ such that the pedal triangle of $V$ with respect to $X$ is similar to $W$. In this case, we call $X$ a ``\emph{pedal center of $V$ with respect to $W$}''. In \cite{ganchev2012points}, Ganchev, Ahmed, and Petkova proved there are $12$ such points for any pair of triangles, and in \cite{chen2015tripod}, Chen and Lourie gave a different proof. This paper consists of two parts. In the first part, we prove the existence of pedal centers between any two triangles and study an invariant obtained from the proof. In the second part, for integer $n\ge 3$, we will define an equivalence relation on $n$-gons to study the pedal relations.

\begin{definition}
In $\mathbb{R}^2$: 
\begin{enumerate}
    \item For points $V_1,V_2$, the line passes through $V_1,V_2$ is denoted by $L(V_1,V_2)$. The segment with endpoints $V_1,V_2$ and its length are both denoted by $\ell(V_1, V_2)$.
    \item For integer $n\ge 3$, an $n$-gon is an $n$-tuple of points in $\mathbb{R}^2$. The $i^{\rm{th}}$ component $V_i$ of an $n$-gon is called the $i^{\rm{th}}$ vertex of this polygon, and in this case such $n$-gon is denoted by $(V_1 V_2\cdots V_n)$. Moreover, we define $V_{i}=V_{i(mod\ n)}$ for all integer $i$.
    \item  For a polygon $V=(V_1V_2\cdots V_n)$ and a point $X$, the \textit{pedal polygon} of $V$ with respect to $X$ is the polygon $\mathcal{P}(V,X)=(P_1P_2\cdots P_n)$ whose vertex $P_i$ is the foot of perpendicular from $X$ to $L(V_i,V_{i+1})$ for all $i\in\{1,2,\cdots,n\}$ (see figure $1$).
    \item For a polygon $V=(V_1V_2\cdots V_n)$ and a point $X$, the \textit{antipedal polygon} of $V$ with respect to $X$ is the polygon $\mathcal{AP}(V,X)=(A_1A_2\cdots A_n)$ where $L(A_i,A_{i+1})$ is perpendicular to $L(X,V_{i})$ and passes through point $V_i$ for all $i\in\{1,2,\cdots,n\}$ (see figure $2$).
\end{enumerate}
\end{definition}    

\begin{figure}[htp]
\begin{center}
\begin{minipage}[t]{0.48\textwidth}
\centering
\includegraphics[width=3.5cm]{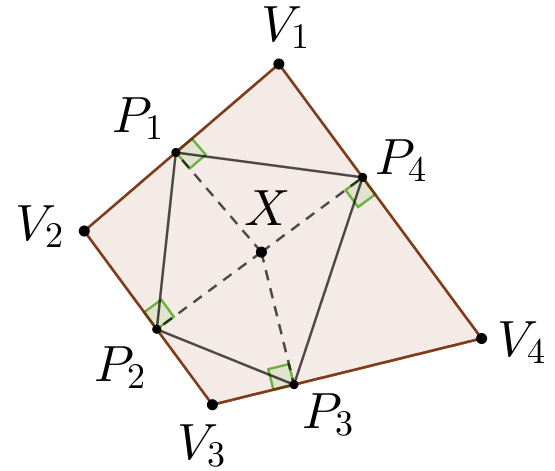}
\caption{Pedal polygon}
\end{minipage}
\begin{minipage}[t]{0.48\textwidth}
\centering
\includegraphics[width=3.3cm]{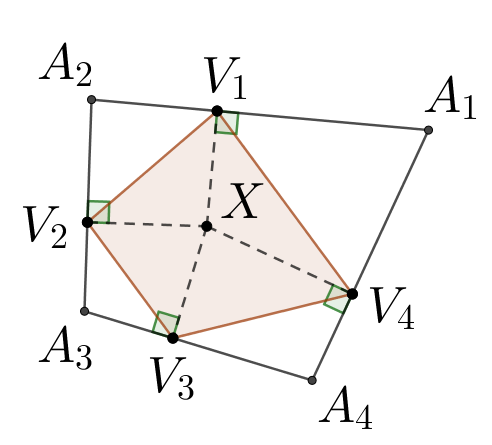}
\caption{Antipodel polygon}
\end{minipage}
\end{center}
\end{figure}

%------------------------------------------------------------------------------------------------------------------------------------------------------------------------------------------------------------------------
\section{Existence of pedal centers for triangles}
\begin{definition}
    Let $\theta\in\mathbb{R}$ and let $V=(V_1V_2\cdots V_n),\  W=(W_1W_2\cdots W_n)$ be $n$-gons. The \textit{outer $n$-gon} $\mathcal{O}(W,V,\theta)=(O_1 O_2\cdots O_n)$ of $W$ with respect to $V$ and $\theta$ is characterized by the following conditions (see figure $3$):
    \begin{enumerate}
        \item $\angle{O_i}=\angle{V_i}$ for all $i\in\{1,2,\cdots,n\}$.
        \item $L(O_i,O_{i+1})$ passes through $W_i$ for all $i\in\{1,2,\cdots,n\}$.
        \item Angle between vector $e_1=(1,0)$ and $\overrightarrow{O_1 O_n}$ is $\theta$.
    \end{enumerate}
    Also, we define the \textit{inner $n$-gon} $\mathcal{I}(W,V,\theta)=(I_1 I_2\cdots I_n)$ of $W$ with respect to $V$ and $\theta$ to be the $n-$gon satisfies the following conditions (see figure $4$):
    \begin{enumerate}
        \item $L(I_i,I_{i+1})$ is perpendicular to $L(O_i,O_{i+1})$ for all $i\in\{1,2,\cdots,n\}$.
        \item $L(I_i,I_{i+1})$ passes through $W_i$ for all $i\in\{1,2,\cdots,n\}$.
    \end{enumerate}
\end{definition}

\begin{figure}[htp]
\begin{center}
\begin{minipage}[t]{0.48\textwidth}
\centering
\includegraphics[width=5.7cm]{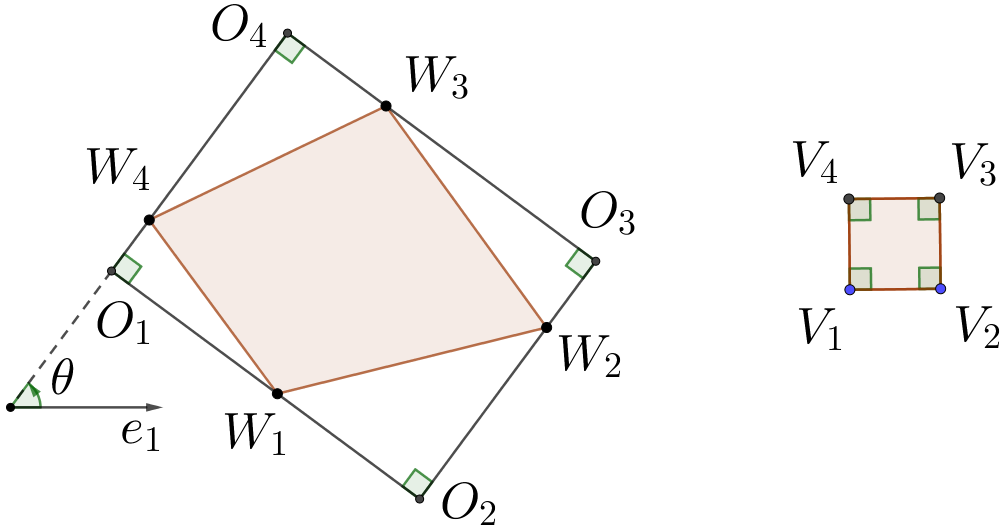}
\caption{Outer polygon $\mathcal{O}(W,V,\theta)$}
\end{minipage}
\begin{minipage}[t]{0.48\textwidth}
\centering
\includegraphics[width=5.7cm]{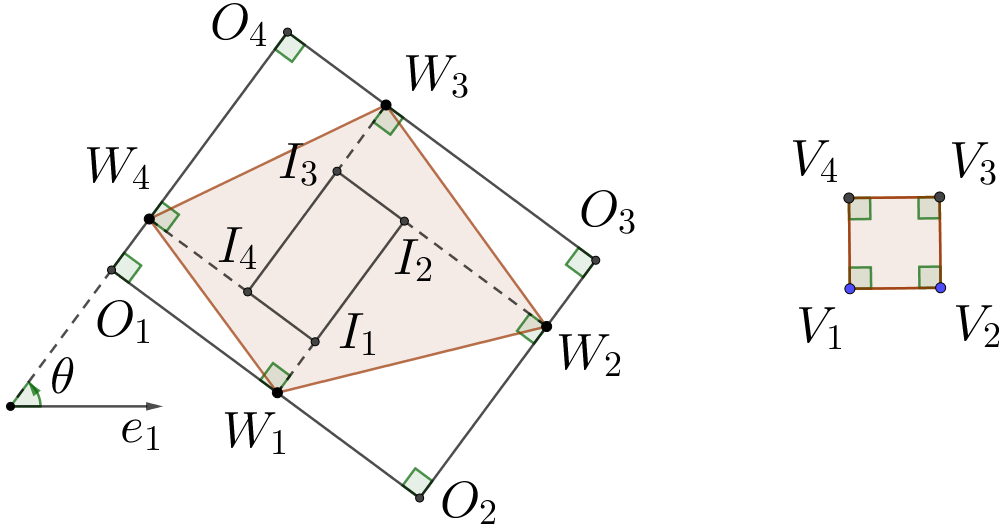}
\caption{Inner polygon $\mathcal{I}(W,V,\theta)$}
\end{minipage}
\end{center}
\end{figure}

\begin{definition}
    For an $n$-gon $V=(V_1V_2\cdots V_n)$ where $V_i=(x_i,y_i)$ for all $i\in\{1,2,\cdots,n\}$, the area of $V$ is
    \[\mathcal{A}(V):=\frac{1}{2}\sum_{i=1}^{n}\det\left(V_i, V_{i+1}\right)=\frac{1}{2}\sum_{i=1}^{n}\det\begin{pmatrix}x_i&x_{i+1}\\y_i&y_{i+1}\end{pmatrix}.\]
    If $\mathcal{A}(V)$ is positive, we say $V$ is oriented counterclockwise; if $\mathcal{A}(V)$ is negative, we say $V$ is oriented clockwise.
\end{definition}

\begin{definition}
    In $\mathbb{R}^2$, for a line $L$ and a unit normal vector $\overrightarrow{v}$ of $L$, the \textit{signed distance} from a point $X$ to $L$ with respect to $\overrightarrow{v}$ is 
    \[d(X,L,{\overrightarrow{v}})=\overrightarrow{v} \cdot \overrightarrow{XY}\]
    where $Y$ is an arbitrary point on $L$.
\end{definition}

\begin{lemma}
    \label{lemma:signed distance}
    Given a triangle $V=(V_1V_2V_3)$ with counterclockwise orientation and a point $X$. If  $\overrightarrow{v_1},\overrightarrow{v_2},\overrightarrow{v_3}$ are outward unit normal vectors of $V$, then the area of $V$ is
    \[\mathcal{A}(V)=\frac{1}{2}\sum_{i=1}^3 {\ell(A_i,A_{i+1})}\cdot d(X,L(V_i,V_{i+1}),{\overrightarrow{v_i}}).\]
\end{lemma}  
    
\begin{theorem}
\label{thm:pedal center}
    For triangles $V=(V_1V_2V_3),W=(W_1W_2W_3)$, there is a point $X$ such that the pedal triangle $\mathcal{P}(V,X)$ is similar to $W$.
\end{theorem}

\begin{proof}
   Without loss of generality, let the coordinates of $W$ are $W_1(0,0)$, $W_2(1,0)$ and $W_3(a,b)$. By the law of sines, there is a function $f:\mathbb{R}\to\mathbb{R}$ such that the inner polygon $\mathcal{I}(W,V,\theta)=(I_1I_2I_3)$ satisfies
    \[\ell(I_2,I_3)=f(\theta)\sin{\angle{V_1}},\ \ell(I_1,I_3)=f(\theta)\sin{\angle{V_2}},\ \textrm{and}\ \ell(I_1,I_2)=f(\theta)\sin{\angle{V_3}}.\] Applying \ref{lemma:signed distance} with $X=W_1(0,0)$, we know the area of $\mathcal{I}(W,V,\theta)$ is
    \[\frac{1}{2}f(\theta)[\sin{\angle{V_1}\cdot\cos(\theta-\angle{V_1}-\angle{V_2})}+\sin{\angle{V_2}\cdot(a\cos{\theta}+b\sin{\theta})}]=:\frac{1}{2}f(\theta)\mathcal{M}(\theta).\]
    Note that $\mathcal{M}(\theta)=\mathbb{M}'(\theta)$, where
    \[\mathbb{M}(\theta)=\sin{\angle{V_1}\cdot\sin(\theta-\angle{V_1}-\angle{V_2})}+\sin{\angle{V_2}\cdot(a\sin{\theta}-b\cos{\theta})}.\] Since $\mathbb{M}(\theta)$ is a differentiable periodic function, it has at least two extreme points. Thus, there exists $\theta_0\in\mathbb{R}$ such that $\mathcal{M}(\theta_0)=\mathbb{M}'(\theta_0)=0$, and hence $\mathcal{A}(\mathcal{I}(W,V,\theta_0))$ vanishes which implies the inner triangle degenerates to a point. 
\end{proof}

\vspace{0.2cm}

\begin{remark}
The idea of proof in \ref{thm:pedal center} is adapted from \cite{tabachnikov1995four} in which such technique is applied to find tripod configurations of a curve.
\end{remark}

\vspace{0.2cm}

\begin{remark}
    From the proof of theorem \ref{thm:pedal center}, we also find a relation between the area of outer and inner polygons.
\end{remark}

\vspace{0.2cm}

\begin{theorem}
\label{thm:invariant of area}
    For $n$-gons $V=(V_1 V_2\cdots V_n),W=(W_1 W_2\cdots W_n)$, there is a constant $c$ such that 
    \[\mathcal{A}(\mathcal{O}(W,V,\theta))+\mathcal{A}(\mathcal{I}(W,V,\theta))=c \quad \te{for $\theta\in\mathbb{R}$}.\]
\end{theorem}

\begin{proof}
Let the coordinates of $W$ be $W_i(x_i,y_i)$ for all $i\in\{1,2,\cdots,n\}$. It is enough to show that $\mathcal{A}(O O_i O_{i+1})+\mathcal{A}(O I_i I_{i+1})$ is constant in $\theta\in\mathbb{R}$ for all $i\in\{1,2,\cdots,n\}$ where $O$ is the origin $(0,0)$, $(O_1O_2\cdots O_n)=\mathcal{O}(W,V,\theta)$ and $(I_1I_2\cdots I_n)=\mathcal{I}(W,V,\theta)$. Note that the slope of $L(O_i,O_{i+1})$ is $s_i=\tan(\theta-\sum_{j=1}^i \angle{V_j})$ and the slope of $L(I_i,I_{i+1})$ is $t_i=-1/{s_i}$ (see figure $5$). Hence we obtain the parametrizations
\[O_i(\frac{s_{i-1}x_{i-1}-y_{i-1}-s_i x_i+y_i}{s_{i-1}-s_i},\frac{-s_{i-1}s_i x_i+s_{i-1}y_i+s_{i-1}s_i x_{i-1}-s_i y_{i-1}}{s_{i-1}-s_i}) \quad\textrm{and}\]
\[I_i(\frac{t_{i-1}x_{i-1}-y_{i-1}-t_i x_i+y_i}{t_{i-1}-t_i},\frac{-t_{i-1}t_i x_i+t_{i-1}y_i+t_{i-1}t_i x_{i-1}-t_i y_{i-1}}{t_{i-1}-t_i}).\]
Therefore, for $\theta\in\mathbb{R}$, $M(\theta):=\mathcal{A}(O O_i O_{i+1})+\mathcal{A}(O I_i I_{i+1})$ satisfies
\begin{equation*}
    \begin{aligned}
    2M(\theta)&=\det(O_i,O_{i+1})+\det(I_{i},I_{i+1})\\
    &=(x_i^2+y_i^2)\frac{\sin{(\angle{V_i}+\angle{V_{i+1}})}}{\sin{\angle{V_i}}\sin{\angle{V_{i+1}}}}+\sum_{j=i}^{i+1}  \left[(x_{j-1}y_j-x_j y_{j-1})+\frac{-x_{j-1}x_j -y_{j-1}y_j}{\tan{\angle{V_j}}}\right] 
    \end{aligned}
\end{equation*}
which implies $M(\theta)$ is a constant and the proof is completed.
\end{proof}

\begin{figure}[htp]
    \centering
    \includegraphics[width=9cm]{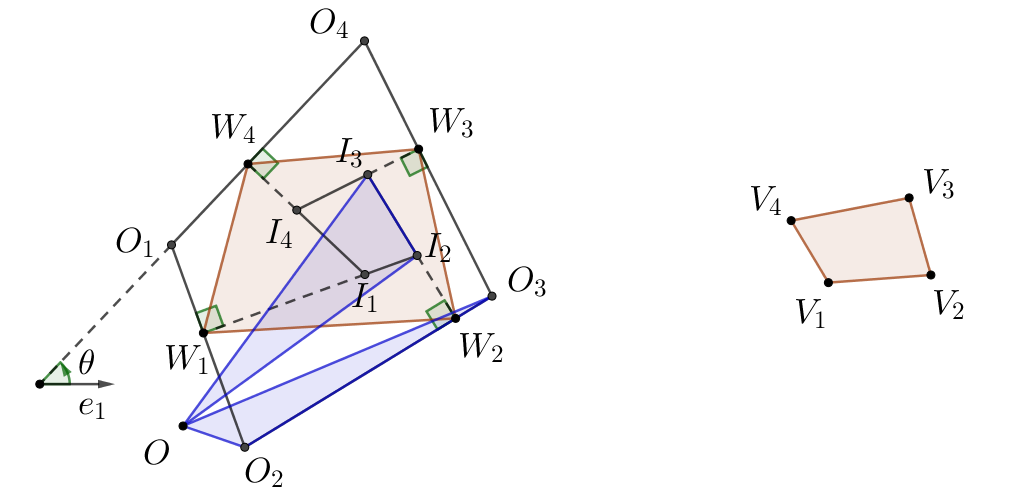}
    \caption{$\mathcal{A}(O O_i O_{i+1})+\mathcal{A}(O I_i I_{i+1})$}
    \label{fig:recursion of N(v)}
\end{figure}

\newpage
\begin{lemma}
\label{lem:similarity of O and I}
For any $n$-gons $V=(V_1V_2\cdots V_n)$, $W=(W_1W_2\cdots W_n)$, if $\mathcal{I}(W,V,\theta_0)$ is a point for some $\theta_0\in\mathbb{R}$, then $\mathcal{O}(W,V,\theta)$ is similar to $\mathcal{I}(W,V,\rho)$ for all $\theta,\rho\in\mathbb{R}$.
\end{lemma}

\begin{proof}
For $\theta\in\mathbb{R}$, let $\mathcal{O}(W,V,\theta)=(O_1O_2\cdots O_n)$, $\mathcal{I}(W,V,\theta)=(I_1I_2\cdots I_n)$ where $O_i(\theta)$, $I_i(\theta)$ are functions of $\theta$ for all $i\in\{1,2,\cdots,n\}$. Since $\angle{O_i(\theta)}=\angle{I_i(\theta)}=\angle{V_i}$ for all $\theta\in\mathbb{R}$, the image of $O_i:\mathbb{R}\to\mathbb{R}^2$ is a circle $C_i$ whose diameter is $\ell(O_i(\theta_1),I_i(\theta_1))$ for any $\theta_1\in\mathbb{R}$. Moreover, $\mathcal{I}(W,V,\theta_0)$ is a point, so $\bigcap_i{C_i}\neq\phi$. If $X\in\bigcap_i{C_i}$, it is easy to see that $(X O_i O_{i+1})(\theta)$ is similar to $(X I_i I_{i+1})(\rho)$ for all $\theta,\rho\in\mathbb{R}$ and $i\in\{1,2,\cdots,n\}$, and hence $\mathcal{O}(W,V,\theta)$ is similar to $\mathcal{I}(W,V,\rho)$ (see figure $6$).
\end{proof}

\begin{figure}[htp]
\begin{center}
\begin{minipage}[t]{0.48\textwidth}
\centering
\includegraphics[width=5.8cm]{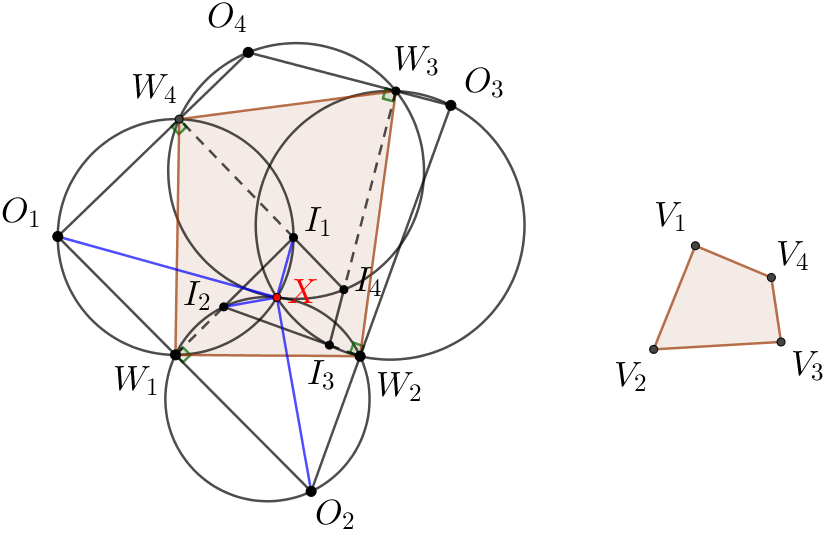}
\caption{$(X O_i O_{i+1})$ and $(X I_i I_{i+1})$}
\end{minipage}
\begin{minipage}[t]{0.48\textwidth}
\centering
\includegraphics[width=5.8cm]{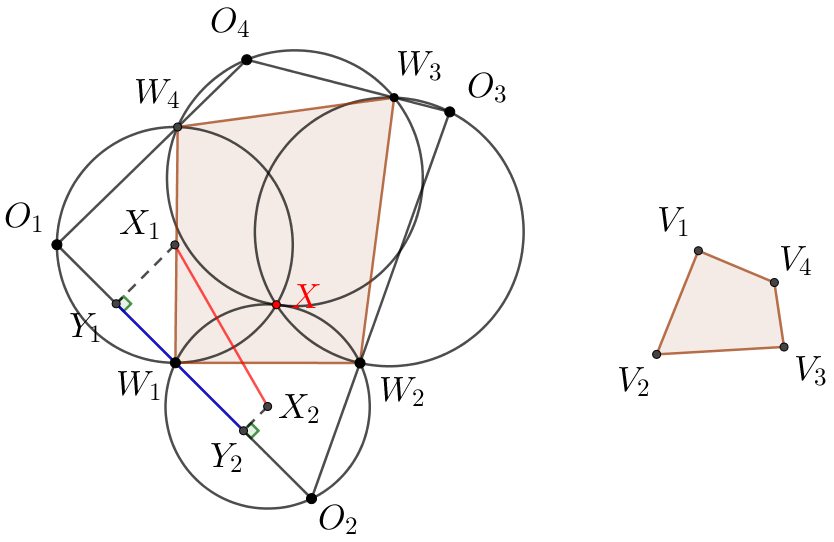}
\caption{Construction of $X_i$ and $Y_i$}
\end{minipage}
\end{center}
\end{figure}

\begin{theorem}
For any $n$-gons $V=(V_1V_2\cdots V_n)$, $W=(W_1W_2\cdots W_n)$, if  $\mathcal{I}(W,V,\theta_0)$ is a point for some $\theta_0\in\mathbb{R}$, then 
\[|\mathcal{A}(\mathcal{O}(W,V,\theta_0))|=\max{\{|\mathcal{A}(\mathcal{O}(W,V,\theta))|:\theta\in\mathbb{R}\}}.\]
\end{theorem}

\begin{proof}
By \ref{thm:invariant of area}, it is enough to show if $|\mathcal{A}(\mathcal{O}(W,V,\theta_0))|=\max{\{|\mathcal{A}(\mathcal{O}(W,V,\theta))|:\theta\in\mathbb{R}\}}$ for some $\theta_0\in\mathbb{R}$, then $\mathcal{I}(W,V,\theta_0)$ is a point. Lemma \ref{lem:similarity of O and I} states that $\mathcal{O}(W,V,\theta)$ is similar to $\mathcal{O}(W,V,\rho)$ for all $\theta,\rho\in\mathbb{R}$. Therefore if $|\mathcal{A}(\mathcal{O}(W,V,\theta_0))|=\max{\{|\mathcal{A}(\mathcal{O}(W,V,\theta))|:\theta\in\mathbb{R}\}}$, then $\ell(O_1,O_2)(\theta_0)\ge \ell(O_1,O_2)(\theta)$ for all $\theta\in\mathbb{R}$. Using notations in \ref{lem:similarity of O and I}, let the center of circle $C_i$ be $X_i$ for all $i\in\{1,2,\cdots,n\}$. We obtain
\[\ell(Y_1,Y_2)=\frac{1}{2}\ell(O_1,O_2)\]
where $Y_i$ is the foot of perpendicular from $X_i$ to $L(O_1,O_2)$ for $i\in\{1,2\}$. Therefore 
\[\ell(Y_1,Y_2)(\theta_0)\ge \ell(Y_1,Y_2)(\theta)\]
for all $\theta\in\mathbb{R}$. It implies $\ell(Y_1,Y_2)(\theta_0)$ is parallel to $\ell(X_1,X_2)$ and hence $L(I_1,I_2)(\theta_0)$ passes through $X$ (see figure $7$). Similar argument shows that $L(I_i,I_{i+1})(\theta_0)$ passes through $X$ for $i\ge 2$, that is, $\mathcal{I}(W,V,\theta_0)$ is a single point.
\end{proof}

\begin{corollary}
\label{cor:extreme of area of pedal triangles}
For triangles $V$, $W$ with positive orientation, if $\mathcal{A}(I(W,V,\theta_0))=0$ for some $\theta_0\in\mathbb{R}$, then 
\[\mathcal{A}(\mathcal{O}(W,V,\theta_0))=\max\{\mathcal{A}(\mathcal{O}(W,V,\theta)):\theta\in\mathbb{R}\}.\]
\end{corollary}

\begin{remark}
In general, for $n\ge 4$ corollary \ref{cor:extreme of area of pedal triangles} may fail to hold, because in this case, $\mathcal{A}(\mathcal{I}(W,V,\theta))=0$ is not equivalent to $\mathcal{I}(W,V,\theta)$ being a point.
\end{remark}

\newpage
\section{Equivalence class about Pedal Polygons}
In this section, we will assume that any successive three vertices of a polygon can not be collinear.

\begin{remark}
    From \ref{thm:pedal center}, we see that for any pair of triangles, there is a point $X$, pedal center, such that one is the pedal triangle of the other with respect to $X$. However, it is not true for $n$-gons with $n>3$. For example, the pedal quadrilateral of a square with respect to any point must have diagonals orthogonal to each others. Therefore, to study the pedal relations for any polygons, we must use a weaker definition.
\end{remark}

\begin{definition}
For an $n$-gon $V=(V_1 V_2\cdots V_n)$ and a sequence of points $S=\{X_i\}_{i=1}^m$, we define $\mathcal{P}^{(1)}(V,S)=\mathcal{P}(V,X_1)$ and inductively \[\mathcal{P}^{(i)}(V,S)=\mathcal{P}(\mathcal{P}^{(i-1)}(V,S),X_i)=(P^{(i)}_1 P^{(i)}_2\cdots P^{(i)}_n)\] 
for all integer $i\ge 1$ (see figure $8$). $\mathcal{P}^{(m)}(V,S)$ is called the \textit{pedal polygon of $V$ with respect to $S$} and denoted by $\mathcal{P}(V,S)$. 
\end{definition}

\begin{figure}[htp]
\begin{center}
\begin{minipage}[t]{0.48\textwidth}
\centering
\includegraphics[width=5.3cm]{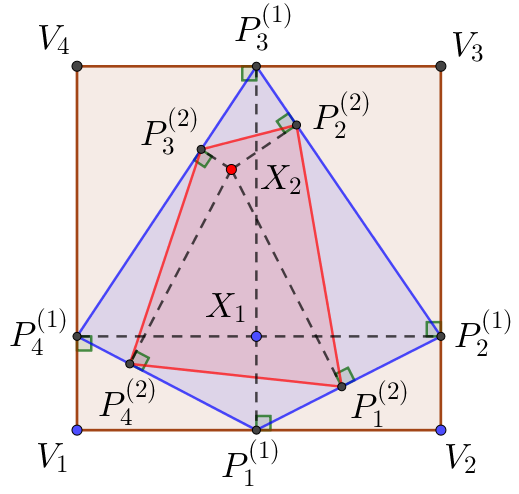}
\caption{Pedal polygon $\mathcal{P}(V,S)$}
\end{minipage}
\begin{minipage}[t]{0.48\textwidth}
\centering
\includegraphics[width=6cm]{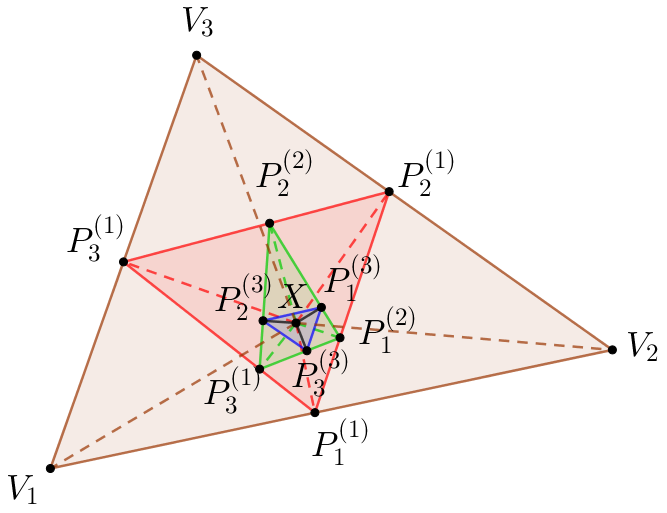}
\caption{$\mathcal{P}(V,\{X\}_{i=1}^3)$ for a triangle}
\end{minipage}
\end{center}
\end{figure}

\begin{lemma}
\label{lem:similarity of itself}
For an $n$-gon $V=(V_1 V_2\cdots V_n)$ and any point $X$, if $S=\{X_i\}_{i=1}^n$ where $X_i=X$ for all $i\in\{1,2,\cdots,n\}$, then $\mathcal{P}(V,S)$ is similar to $V$. 
\end{lemma}

\begin{proof}
It is easy to see
\begin{equation*}
\begin{cases}
\angle{X P^{(0)}_i P^{(0)}_{i+1}}=\angle{X P^{(1)}_{i-1} P^{(1)}_{i}}=\angle{X P^{(2)}_{i-2} P^{(2)}_{i-1}}=\cdots =\angle{X P^{(n)}_{i} P^{(n)}_{i+1}}\\
\angle{X P^{(0)}_{i+1} P^{(0)}_{i}}=\angle{X P^{(1)}_{i+1} P^{(1)}_{i}}=\angle{X P^{(2)}_{i+1} P^{(2)}_{i}}=\cdots =\angle{X P^{(n)}_{i+1} P^{(n)}_{i}}
\end{cases}
\label{cond:even}
\end{equation*}
for all $i\in\{1,2,\cdots,n\}$ which implies $(X P^{(0)}_i P^{(0)}_{i+1})=(X V_i V_{i+1})$ is similar to $(X P^{(n)}_i P^{(n)}_{i+1})$ for all $i\in\{1,2,\cdots,n\}$ (see figure $9$). Therefore $\mathcal{P}^{(n)}(V,S)$ is similar to $V$.
\end{proof}

\begin{definition}
For $n$-gons $V$ and $W$, we say $V$ and $W$ are \textit{pedal equivalent} if there is a sequence of points $S$ such that $\mathcal{P}(V,S)$ is similar to $W$. In this case, we denote $V\perp W$.  
\end{definition}

\begin{theorem}
\label{thm:equivalence relation}
The relation $\perp$ on set of $n$-gons is an equivalence relation. 
\end{theorem}

\begin{proof}
It is easy to see $\perp$ is transitive and lemma \ref{lem:similarity of itself} implies $\perp$ is reflexive. For symmetry, if $V\perp W$ then there is a sequence $S=\{X_i\}_{i=1}^m$ such that $\mathcal{P}(V,S)$ is similar to $W$. Notice that $ \mathcal{P}( \mathcal{P}^{(m-1)}(V,S) ,X_m)$ is similar to $W$. Thus in view of \ref{lem:similarity of itself}, we see $W \perp  \mathcal{P}^{(m-1)}(V,S)$. Repeating this process and make use of transitivity of $\perp$, we conclude $W \perp V$.
\end{proof}

\newpage
\begin{corollary}%三角形僅有一類
Let $G(3)$ be the set of triangles, then $|G(3)/{\perp}|=1$.
\end{corollary}

\begin{proof}
Theorem \ref{thm:equivalence relation} states that ``$\perp$'' is an equivalence relation on $G(3)$, and \ref{thm:pedal center} implies $V\perp W$ for all $V,W\in G(3)$.
\end{proof}

\begin{lemma}
\label{lem:square}
For any square $V=(V_1 V_2 V_3 V_4)$ and rectangle $W=(W_1 W_2 W_3 W_4)$, we have $V\perp W$.
\end{lemma}

\begin{proof}
For a point $X_1$ on the perpendicular bisector of $\ell(V_1,V_4)$, without loss of generality, assume $X_1=(0,0)$, $V_1=(-1,t)$, $V_2=(-1,t-2)$, $V_3=(1,t-2)$, $V_4=(1,t)$, $P^{(1)}_1=(-1,0)$, $P^{(1)}_3=(1,0)$, $P^{(1)}_4=(0,t)$, $P^{(1)}_2=(0,t-2)$ for some $0<t<2$ where $(P^{(1)}_1 P^{(1)}_2 P^{(1)}_3 P^{(1)}_4)=\mathcal{P}(V,X_1)$ (see figure $10$). It is obvious that there is a unique point $X_2=(0,2(t-1)/(-t^2+2t+1))$ such that $\mathcal{P}(V,\{X_i\}_{i=1}^2)$ is a rectangle (see figure $11$). Moreover, the ratio of length and width of $\mathcal{P}(V,\{X_i\}_{i=1}^2)$ is $r(t)=-t(t-2)$. Note that the range of function $r:(0,2)\to\mathbb{R}$ is $(0,1]$ which implies that for any rectangle $W$, there is a $t_0\in(0,2)$ such that $\mathcal{P}(V,\{X_i\}_{i=1}^2)$ is similar to $W$.
\end{proof}

\begin{figure}[htp]
\begin{center}
\begin{minipage}[t]{0.48\textwidth}
\centering
\includegraphics[width=4.5cm]{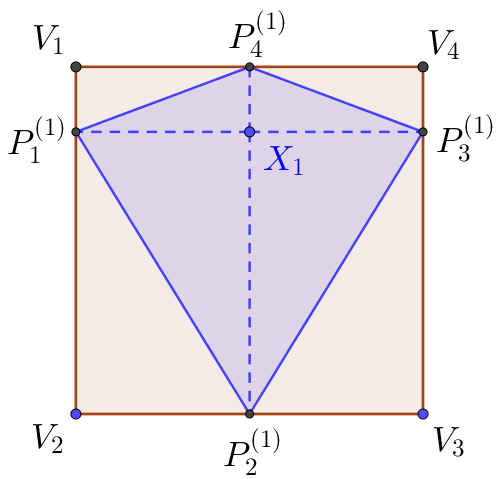}
\caption{Pedal polygon $\mathcal{P}(V,X_1)$}
\end{minipage}
\begin{minipage}[t]{0.48\textwidth}
\centering
\includegraphics[width=4.5cm]{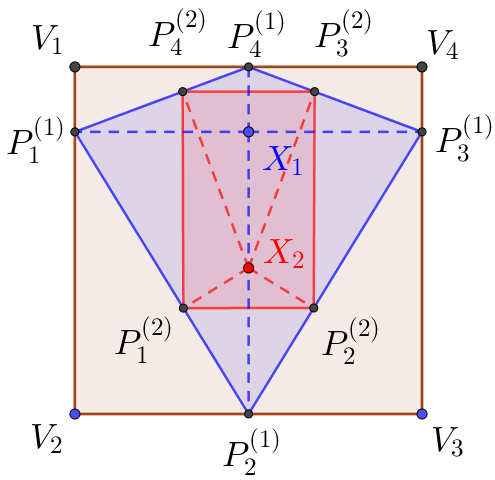}
\caption{Pedal polygon $\mathcal{P}^{(2)}(V,S)$}
\end{minipage}
\end{center}
\end{figure}

\begin{lemma}
\label{lem:kite}
For any kite $K$, isosceles trapezoid $T$, and square $V$, $K\perp V$ and $T\perp V$.
\end{lemma}

\begin{proof}
Any kite is a pedal polygon of a rectangle, and any isosceles trapezoid is a pedal polygon of some kites (see figure $12$ and $13$). This observation together with \ref{thm:equivalence relation} and \ref{lem:square} finishes the proof.    
\end{proof}

\begin{figure}[htp]
\begin{center}
\begin{minipage}[t]{0.48\textwidth}
\centering
\includegraphics[width=4.3cm]{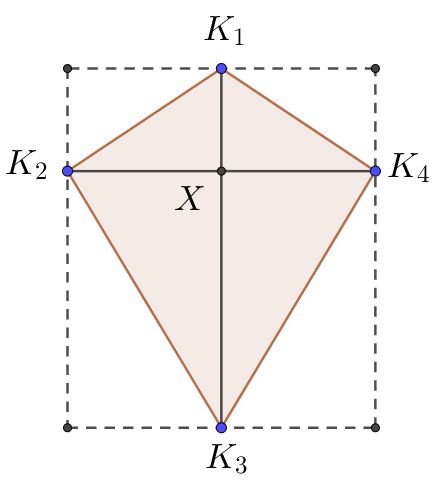}
\caption{Antipedal polygon of $K$}
\end{minipage}
\begin{minipage}[t]{0.48\textwidth}
\centering
\includegraphics[width=4.1cm]{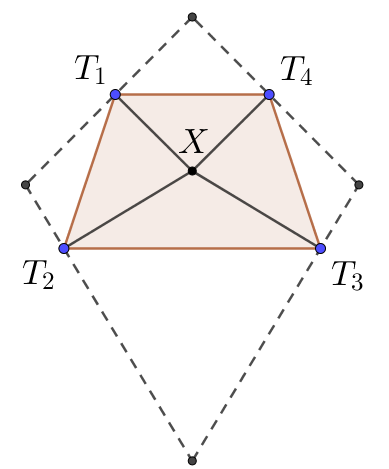}
\caption{Antipedal polygon of $T$}
\end{minipage}
\end{center}
\end{figure}

\begin{remark}
Lemma \ref{lem:square} and \ref{lem:kite} state that all rectangles, kites, and isosceles trapezoids are in the  same equivalence class $[V]$ where $V$ is a square. 
\end{remark}

\begin{theorem}
Let $G(4)$ be the set of quadrilaterals, then $|G(4)/{\perp}|=1$.
\end{theorem}

\begin{proof}
For a simple quadrilateral $W=(W_1 W_2 W_3 W_4)$, there is a point $X$ on $L(W_1,W_3)$ such that the angle between $\overrightarrow{X W_2}$ and $\overrightarrow{W_1 W_3}$ is equal to the one between $\overrightarrow{X W_4}$ and $\overrightarrow{W_1 W_3}$ so that antipedal polygon $\mathcal{AP}(W,X)$ is an isosceles trapezoid (see figure $14$). Finally, for a non-simple quadrilateral $W'=(W'_1 W'_2 W'_3 W'_4)$, let $I$ be the self-intersection point. Notice that there is some $i$ such that $(I W'_i W'_{i+1})$ is a non-degenerated triangle and let $Y$ be its incenter. Then $\mathcal{P}(W',Y)$ is a simple quadrilateral, and the problem is reduced to the case for simple quadrilaterals (see figure $15$). 
\end{proof}

\begin{figure}[htp]
\begin{center}
\begin{minipage}[t]{0.38\textwidth}
\centering
\includegraphics[width=4.2cm]{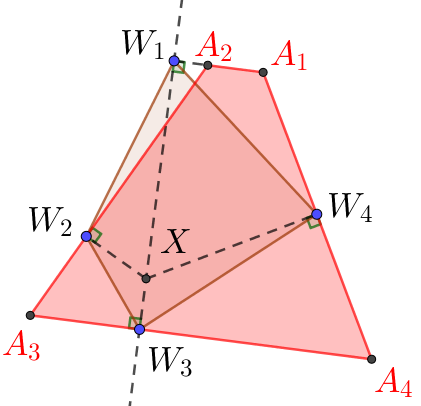}
\caption{Antipedal isosceles trapezoid}
\end{minipage}
\begin{minipage}[t]{0.48\textwidth}
\centering
\includegraphics[width=4.3cm]{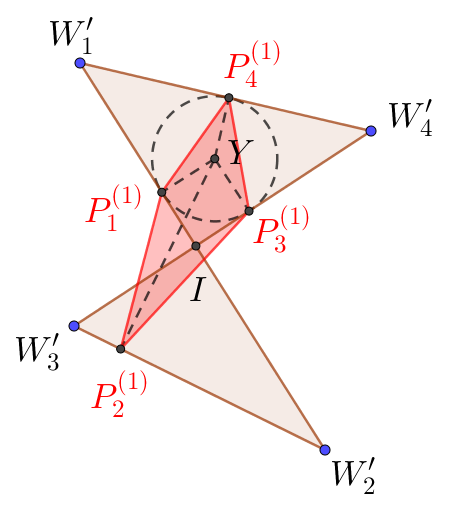}
\caption{Simple pedal polygon}
\end{minipage}
\end{center}
\end{figure}

\begin{conjecture}
For integer $n\ge 5$, if $G(n)$ is the set of $n$-gons, then $|G(n)/{\perp}|=1$.
\end{conjecture}

%----------------------------------------------------------------------------------------------------------
\medskip
%Sets the bibliography style to UNSRT and imports the 
%bibliography file "samples.bib".
\bibliographystyle{alpha}
\bibliography{references}
%----------------------------------------------------------------------------------------------------------

\end{document}